\documentclass[letterpaper, 10 pt, conference]{ieeeconf}  

\IEEEoverridecommandlockouts                              

\usepackage{cite}

\usepackage[subpreambles=true]{standalone}
\usepackage{import}

  \usepackage{pgfplots}
\pgfplotsset{compat=newest}
\usetikzlibrary{plotmarks}
\usetikzlibrary{arrows.meta}
\usepgfplotslibrary{patchplots}
\usepackage{grffile}
\newlength\fwidth

\usepackage{graphics} 
\usepackage{epsfig} 
\usepackage{mathptmx} 
\usepackage{times} 
\usepackage{amsmath} 
\usepackage{amssymb}  
\usepackage[hidelinks]{hyperref} 
\usepackage{verbatim} 
\usepackage[ruled,vlined,noend]{algorithm2e} 
\usepackage{xcolor} 
\usepackage{nicefrac}

\usepackage{caption}
\usepackage{subcaption}

\usepackage{makecell}

\usepackage{etoolbox}
\makeatletter
\patchcmd{\@algocf@start}
  {-1.5em}
  {0pt}
  {}{}
\makeatother
\title{\LARGE \bf
      Fixed-Order $\mathcal{H}_2$-Conic Control*
}

\author{Ethan J. LoCicero$^1$ and Leila Bridgeman$^1$
\thanks{*\copyright 2021 IEEE.  Personal use of this material is permitted.  Permission from IEEE must be obtained for all other uses, in any current or future media, including reprinting/republishing this material for advertising or promotional purposes, creating new collective works, for resale or redistribution to servers or lists, or reuse of any copyrighted component of this work in other works.}%
\thanks{This material is based upon work supported by the National Science Foundation Graduate Research Fellowship Program under Grant No. 1644868 and by the Alfred P. Sloan Foundation Minority Ph.D. Program.}
\thanks{$^{1}$Ethan J. LoCicero (PhD Student) and Leila Bridgeman (assistant Professor) are with the Dept. of Mechanical Eng. and Materials Science at Duke University, Durham, NC, 27708, USA (email: {\tt\small ejl48@duke.edu; ljb48@duke.edu}, phone: 919-660-5310), corresponding author: Ethan J. LoCicero}%
}

\newtheorem{define}{Definition}
\newtheorem{theorem}{Theorem}
\newtheorem{lemma}{Lemma}
\newtheorem{cor}{Corollary}

\makeatletter
\newcommand\Autoref[1]{\@first@ref#1,@}
\def\@throw@dot#1.#2@{#1}
\def\@set@refname#1{
    \edef\@tmp{\getrefbykeydefault{#1}{anchor}{}}%
    \xdef\@tmp{\expandafter\@throw@dot\@tmp.@}%
    \ltx@IfUndefined{\@tmp autorefnameplural}%
         {\def\@refname{\@nameuse{\@tmp autorefname}s}}%
         {\def\@refname{\@nameuse{\@tmp autorefnameplural}}}%
}
\def\@first@ref#1,#2{%
  \ifx#2@\autoref{#1}\let\@nextref\@gobble
  \else%
    \@set@refname{#1}
    \@refname~\ref{#1}
    \let\@nextref\@next@ref
  \fi%
  \@nextref#2%
}
\def\@next@ref#1,#2{%
   \ifx#2@ and~\ref{#1}\let\@nextref\@gobble
   \else, \ref{#1}
   \fi%
   \@nextref#2%
}

\newcommand{\Chat}{\widehat{\mathbf{C}}} 
\newcommand{\Ahat}{\widehat{\mathbf{A}}} 
\newcommand{\Bhat}{\widehat{\mathbf{B}}} 
\newcommand{\I}{\mathbf{I}} 
\newcommand{\mbf}[1]{\mathbf{#1}} 
\newcommand{\bmat}[1]{\begin{bmatrix} #1 \end{bmatrix}} 
\newcommand{\Hcal}{\mathcal{H}} 
\newcommand{\cone}{\mbox{cone}} 
\newcommand{\xhat}{\hat{\mathbf{x}}} 
\newcommand{\A}{\mathbf{A}} 
\newcommand{\B}{\mathbf{B}} 
\newcommand{\C}{\mathbf{C}} 
\newcommand{\D}{\mathbf{D}} 
\newcommand{\x}{\mathbf{x}} 
\newcommand{\bu}{\mathbf{u}} 
\newcommand{\y}{\mathbf{y}} 
\newcommand{\yhat}{\hat{\mathbf{y}}} 
\newcommand{\w}{\mathbf{w}} 
\newcommand{\z}{\mathbf{z}} 
\newcommand{\bP}{\mathbf{P}} 
\newcommand{\Q}{\mathbf{Q}} 
\newcommand{\tr}{\mbox{tr}} 
\newcommand{\K}{\mbf{K}} 
\newcommand{\E}{\mbf{E}} 
\newcommand{\bS}{\mbf{S}} 
\newcommand{\R}{\mbf{R}}
\newcommand{\F}{\mbf{F}}
\newcommand{\Qt}{\widetilde{\Q}}

\newcommand{\dQ}{\delta\Q}
\newcommand{\dK}{\delta\K}
\newcommand{\He}{\mbox{He}}
\newcommand{\M}{\mbf{M}}
\newcommand{\bL}{\mbf{L}}
\newcommand{\W}{\mbf{W}}
\newcommand{\V}{\mbf{V}}
\newcommand{\Pt}{\widetilde{\mbf{P}}}
\newcommand{\dP}{\delta\mbf{P}}
\newcommand{\dPt}{\delta\widetilde{\mbf{P}}}
\newcommand{\X}{\mbf{X}}
\newcommand{\Gcal}{\mathcal{G}}
\newcommand{\At}{\widetilde{\A}}
\newcommand{\Bt}{\widetilde{\B}}
\newcommand{\Ct}{\widetilde{\C}}
\newcommand{\st}{\mbox{s.t.}}
\newcommand{\Htwo}{{\mathcal{H}_2}}
\newcommand{\zero}{\mbf{0}}
\newcommand{\Z}{\mbf{Z}}

\begin{document}

\maketitle
\thispagestyle{empty}
\pagestyle{empty}

\begin{abstract}
$\Hcal_2$-conic controller design seeks to minimize the closed-loop $\Hcal_2$ norm for a nominal linear system while satisfying the Conic Sector Theorem for nonlinear stability. This problem has only been posed with limited design freedom, as opposed to fixed-order design where all controller parameters except the number of state estimates are free variables. Here, the fixed-order $\Hcal_2$-conic design problem is reformulated as a convergent series of convex approximations using iterative convex overbounding. A synthesis algorithm and various initializations are proposed. The synthesis is applied to a passivity-violated system with uncertain parameters and compared to benchmark controller designs.
\end{abstract}


\section{Introduction} \label{introduction}




Recently, input-output (IO) analysis and stability theorems have emerged as useful tools in the fields of robust control, data-driven control, and robust neural networks \cite{Romer2019, Martin2021, Ghodrat2020, Fazlyab2020,Berberich2020, Dong2020, McCourt2020, Turner2020, Polushin2019}. For robust controller synthesis, conic-sector-based methods rely on more general stability results than the familiar Passivity and Small Gain Theorems, which often improves the performance-robustness tradeoff \cite{Bridgeman2014, Sivaranjani2018, Brown2021, Bridgeman2019a, Joshi2000, Joshi2002, Xia2020, Bridgeman2018, Gupta1994}. However, current design methods allow incomplete access to controller parameters. This erodes performance and impedes the addition of secondary objectives like sparsity promotion. Current designs also do not admit reduced order controllers. This paper proposes a fixed-order conic-sector-based design method that allows complete access to controller parameters for a controller of any specified dimension, improving design flexibility and admitting secondary objectives.

Performance versus robustness to model uncertainty is a fundamental tradeoff in automatic feedback control. This can be seen in design methods like $\Hcal_2/\Hcal_\infty$ control \cite{Khargonekar1991} and $\Hcal_2$-SPR \cite{Geromel1997a} control, where the $\Hcal_2$ performance objective is optimized subject to the Small Gain and Passivity Theorems, respectively. These IO stability results guarantee closed-loop stability for nonlinear, uncertain plants.

Passivity and small gain, though widely used, are not always applicable. Several other results -- large gain, $\gamma$-passivity, and passivity indices -- have been developed for cases that passivity and small gain cannot handle. It was recetly shown that these are all special cases of the Extended Conic Sector Theorem \cite{Bridgeman2018}. Zames' Conic Sector Theorem (CST) \cite{Zames1966} predates many IO stability theorems, but it was relatively unused for several decades due to its analytical complexity. Even more general IO results -- like QSR dissipativity \cite{Berberich2020,Brogliato2007} -- present similar analytical challenges. However, new tools have been developed to facilitate conic-sector-based analysis and design. Matrix inequality constraints that impose conic bounds on a linear system were developed using the KYP Lemma in the 90's \cite{Gupta1994}. Since then, several design schemes attempting to minimize the $\mathcal{H}_2$-norm subject to the conic constraint have been proposed \cite{Bridgeman2019a,  Sivaranjani2018, Brown2021, Bridgeman2014, Joshi2000, Joshi2002}. These ``$\Hcal_2$-conic" designs can be thought of as a stepping-stone to future QSR-dissipativity-based designs. However, simpler characterizations involved in conic properties facilitate interpretability and informed design choices.

The primary challenge in $\Hcal_2$-conic design is the nonconvexity of the cubic objective and bilinear constraints. This makes the problem NP-hard, so no method can guarantee polynomial-time convergence to a global minimum  \cite{Toker1995}. The aforementioned design schemes instead restrict design freedom in various ways, resulting in a conservative convex problem or series thereof. The most straightforward method fixes the observer and designs only the state-estimate feedback matrix using a convex heuristic proxy for the $\Hcal_2$ norm \cite{Bridgeman2014}. Other methods enforce a Luenberger structure on the controller and transform controller parameters into implicit variables to achieve convex constraints and a convex overbound on the $\Hcal_2$-norm \cite{Bridgeman2019a, Joshi2000, Joshi2002}. The restricted design space in these methods inevitably reduces performance, and forcing the controller order to match that of the plant is problematic for large systems. Further, incomplete access to controller parameters impedes addition of secondary objectives like sparsity promotion and structured communication, which are important to stability and security problems in distributed design \cite{Eilbrecht2017, Sandell1978, Chowdhury2019}. The present work serves as an intermediate result towards structured $\Hcal_2$-conic designs.

This paper addresses the above limitations by proposing a fixed-order $\Hcal_2$-conic design scheme. In this context, ``fixed-order" means the number of state estimates $n_c$ is fixed, and controller parameters $\Ahat$, $\Bhat$, and $\Chat$ (as defined in Section~\ref{problemstatement}) are accessable design variables. This is as opposed to a ``Luenberger" controller, where the state-estimate dynamics matrix is assumed to have the structure $\Ahat = \A-\B_2\Chat-\Bhat\C_2$. This improved design is achieved by employing iterative convex overbounding (ICO) \cite{Warner2017a,Oliveira2000} to pose the nonconvex problem as a convergent series of convex problems without obfuscating the controller parameters.


\section{Preliminaries} \label{prelim}

\subsection{Notation} \label{notation}

For a square matrix $\bP$, positive definiteness is denoted $\bP > 0$. Related properties (negative definiteness and positive/negative semi-definiteness) are denoted correspondingly. The identity matrix is $\I$, the trace is $\tr(\cdot)$, and $\lambda(\cdot)$ is the eigenvalues. For brevity, $\He[\M] = \M + \M^T$ and astrisks denote duplicate blocks in symmetric matrices. The $L_2$, Frobenius, and $\mathcal{H}_2$ norms are denoted $||\cdot||_2$, $||\cdot||_F$, $||\cdot||_{\mathcal{H}_2}$. Recall $\y \in L_2$ if $||\y||_2^2 = \langle \y,\, \y \rangle = \int_0^\infty \y^T(t)\y(t)dt <\infty$. Further, $\y\in L_{2_e}$ if its truncation to $[0,\,T]$ is in $L_2$ $\forall$ $T\geq 0$, where the truncation of $\y$ is found by multiplying $\y(t)$ by 0 for $t>T$.


\subsection{Review of Conic Sectors} \label{review}
Conic sectors, defined below, are an IO description of a set of operators. They can be viewed as a special case of dissipativity.
\medskip
\begin{define}
\textit{(Interior conic \cite{Bridgeman2014})} A square system, $\mathcal{G} : L_{2e}\rightarrow L_{2e}$, is in the conic sector $[a,\,b]$, where $a<b$ and $0<b$, denoted $\mathcal{G} \in \mbox{cone}[a,\,b]$, if $\forall\; \mathbf{u} \,\in\, L_{2e},\;T\,\in\, \mathbb{R}^+$,
\begin{eqnarray*}
\frac{1}{b}||\mathcal{G}\mathbf{u}||_{2T}^2 + \Big(1+\frac{a}{b}\Big)\, \langle\mathcal{G}\mathbf{u},\,\mathbf{u}\rangle_T - a ||\mathbf{u}||_{2T}^2 \geq \beta,
\end{eqnarray*}
where $\beta$ only depends on initial conditions. It is \textit{strictly} in the conic sector, denoted $\mathcal{G}\in\mbox{cone}(a,\,b)$, if $\mathcal{G} \in \mbox{cone}[a+\delta,\,b-\delta]$, for some small $\delta > 0$.
\end{define}
\medskip


The Conic Sector Lemma (CSL) below shows a matrix inequality expression of conic bounds for linear time invariant (LTI) systems without feedthrough.
\medskip
\begin{lemma} \label{CSL}
\textit{(Conic Sector Lemma \cite{Gupta1994})} A square LTI system, $\mathcal{G}:L_{2e}\rightarrow L_{2e}$, with minimal state-space realization $(\mathbf{A},\,\mathbf{B},\,\mathbf{C},\,\mathbf{0})$ is inside cone $[a,\,b]$, $a < 0 < b < \infty$ if and only if there exists $\mathbf{P} = \mathbf{P}^T > 0$, such that 
\begin{eqnarray}
	&&\bmat{\bP\A+\A^T\bP + \C^T\C &* \\ \B^T\bP - \frac{a+b}{2}\C & ab\I } \leq 0 \label{eqn:CSL} \\
	\mbox{or} && \nonumber\\
	&& \bmat{\bP\A + \A\bP & \bP\B & \C^T \\ * & -\frac{(a-b)^2}{4b}\I & -\frac{(a+b)}{2}\I \\ * & * & -b\I} \leq 0 \label{dialated_conic}\\
	\mbox{or} && \nonumber\\
	&&\bP(\A+\frac{a+b}{2ab}\B\C) + (\A^T+\frac{a+b}{2ab}\C^T\B^T)\bP \nonumber \\
	&&+ (1-\frac{(a+b)^2}{4ab})\C^T\C - \frac{1}{ab}\bP\B\B^T\bP \leq 0. \label{contracted_conic}
\end{eqnarray}
\end{lemma}
Equations~\ref{eqn:CSL}--\ref{contracted_conic} are equivallent. The conic sector with $a=0$ and $b=\infty$ is identical to passvity. In fact, dividing \autoref{eqn:CSL} by $b>0$, setting $a=0$, and taking the limit as $b \rightarrow \infty$  recovers the KYP Lemma for passivity \cite[section 3]{Brogliato2007}.
 


In this paper, the CSL will be used in combination with the CST to ensure closed-loop IO stablity. The CST statement below is more restricted than the original formulation in \cite{Zames1966} but contains all cases needed in this paper.
\medskip
\begin{theorem}
\textit{(CST \cite{Zames1966, Bridgeman2014})} Consider the negative feedback interconnection of two square systems, $\mathcal{G}_1:L_{2e}\rightarrow L_{2e}$ and $\mathcal{G}_2 : L_{2e}\rightarrow L_{2e}$, defined by
$\mathbf{y}_1 = \mathcal{G}_1\mathbf{u}_1$, $\mathbf{y}_2 = \mathcal{G}_2\mathbf{u}_2$,
$\mathbf{u}_1 = \mathbf{r}_1 - \mathbf{y}_2$, and $\mathbf{u}_2 = \mathbf{r}_2 -\mathbf{y}_1$,
where $\mathbf{r}_i, \, \mathbf{u}_i, \, \mathbf{y}_i \, \in \, L_{2e}$ for $i = 1,\,2$ (see \autoref{conicfeedback}). If $\mathcal{G}_1 \in \mbox{cone}[a,\,b]$ for $a < 0 < b $ and $\mathcal{G}_2 \, \in \, \mbox{cone}(-\frac{1}{b},\,-\frac{1}{a})$ then the closed-loop system $\mathbf{y} = \mathcal{G}\mathbf{r}$, where $\mathbf{r}^T = [\mathbf{r}_1^T \; \mathbf{r}_2^T]$ and $\mathbf{y}^T = [\mathbf{y}_1^T \; \mathbf{y}_2^T]$, is input-output stable.
\label{thm:CST} 
\end{theorem}
\medskip

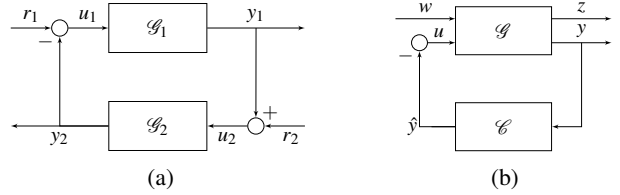
\begin{figure}
     \centering
     \begin{subfigure}[b]{0.47\columnwidth}
         \centering
         \resizebox{4cm}{!}{%
                 \begin{tikzpicture}[auto, node distance=2.5cm,>=latex']
        \node [input, name=input] {};
        \node [sum, right of=input] (sum) {};
        \node [block, right of=sum, xshift=-0.5cm] (plant) {\large $\mathcal{G}_1$};
        \node [output, right of=plant, xshift=1cm] (output) {};
        \node [block, below of=plant, yshift=0.5cm] (feedback) {\large $\mathcal{G}_2$};
        \node [output, left of=feedback, xshift=-1cm] (output2) {};
        \node [sum, right of=feedback, xshift=1cm] (sum2) {};
         \node [input, right of=sum2] (input2) {};
        \draw [draw,->] (input) -- node {\large $r_1$} (sum);
        \draw [draw,->] (input2) -- node[pos=.7, anchor=north west] {\large $r_2$} (sum2);
        \draw [->] (sum) -- node {\large $u_1$} (plant);
        \draw [->] (sum2) -- node[anchor=north]{\large $u_2$} (feedback);
        \draw [->] (plant) -- node [name=y] {\large $y_1$}(output);
        \draw [->] (feedback) -- node [name=y2] {\large $y_2$}(output2);
        \draw [->] (y) -- node[pos=0.95] {\large $+$} node [near end] {} (sum2);
        \draw [->] (feedback) -| node[pos=0.97] {\large $-$} node [near end] {} (sum);
        \end{tikzpicture}
         }
         \caption{}
         \label{conicfeedback}
     \end{subfigure}
     \hfill
     \begin{subfigure}[b]{0.47\columnwidth}
         \centering
         \resizebox{3cm}{!}{%
                 \begin{tikzpicture}[auto, node distance=2.5cm,>=latex']
        \node [block, name=plant] {\large $\mathcal{G}$};
        \node[output, left of=plant, xshift=1cm, yshift=-0.25cm] (g_left_bottom) {};
        \node[output, right of=plant, xshift=-1cm, yshift=-0.25cm] (g_right_bottom) {};
        \node[output, left of=plant, xshift=1cm, yshift=0.25cm] (g_left_top) {};
        \node[output, right of=plant, xshift=-1cm, yshift=0.25cm] (g_right_top) {};
        \node [block, below of=plant, yshift=0.5cm] (feedback) {\large $\mathcal{C}$};
        \node [output, right of=plant, xshift=0.25cm, yshift=-0.25cm] (output) {};
        \draw [->] (g_right_bottom) -- node [name=y] {\large $y$}(output);
        \draw [->] (y) |- (feedback);
        \node [sum, left of=plant, xshift=-0.75cm, yshift=-0.25cm] (inverter) {};
        \draw [->] (feedback) -| node[pos=0.97] {\large $-$} node [near end] {} (inverter);
        \node[output, left of=feedback, xshift=0.5cm] (yhat) {};
        \draw (feedback) -- node[pos=1.35,anchor= east] {\large $\hat y$} (yhat);
        \draw [->] (inverter) -- node[pos=.4] {\large $u$} (g_left_bottom);
        \node[output, left of=plant, xshift=-0.25cm, yshift=0.25cm] (disturbance) {};
        \draw [->] (disturbance) -- node {\large $w$} (g_left_top);
        \node[output, right of=plant, xshift=0.25cm, yshift=0.25cm] (cout) {};
        \draw [->] (g_right_top) -- node {\large $z$} (cout);
        
        \end{tikzpicture}
         }
         \caption{}
         \label{fig_cl}
     \end{subfigure}
 \caption{Closed loops with (a) additive disturbances and (b) general disturbances and objectives.}
\end{figure}


If a system is in any conic sector, it is in infinitely many conic sectors. An appropriate conic sector for controller synthesis is one that maximizes controller design space while robustly characterizing plant uncertainty. Such a sector can be identified analytically or with the CSL\cite{Bridgeman2014}, frequency domain methods, \cite{Xia2020}, or data-driven methods \cite{Romer2019, Martin2021}.


\section{Problem Statement} \label{problemstatement}
Consider the system in Figure~\ref{fig_cl} with plant and controller
\begin{eqnarray*}
	\Gcal : \begin{cases}
	\dot{\x} = \A\x + \B_1\w + \B_2\bu \\
	\z = \C_1\x + \D_{12}\bu \\
	\y = \C_2\x + \D_{21}\w
	\end{cases},
&&
	\mathcal{C} : \begin{cases}
	\dot{\xhat} = \Ahat\xhat + \Bhat\y \\
	\yhat = \Chat\xhat
	\end{cases}.
\end{eqnarray*}
Here, $\x \in \mathbb{R}^n$, $\bu \in \mathbb{R}^m$, $\y \in \mathbb{R}^m$, $\w \in \mathbb{R}^p$, $\z \in \mathbb{R}^q$, $\xhat \in \mathbb{R}^{n_c}$, $\yhat \in \mathbb{R}^m$ are the states, inputs, measured outputs, distrubances, controlled outputs, state-estimates, and controller outputs, respectively.
The $\Htwo$-conic problem is
\begin{eqnarray} \label{general}
\! \! \! \!	\min_{\Ahat,\Bhat,\Chat} \; \; \; \; J = ||\Gcal_{cl}||_{\Htwo}^2 &\st& (\Ahat,\Bhat,\Chat,\zero) \in \cone(a_c,b_c),
\end{eqnarray}
where $\Gcal_{cl}$ is the closed-loop state space realization 
\begin{eqnarray*}
		\left[
	\begin{array}{c | c}
		\A_{cl} & \B_{cl} \\
		\hline
		\C_{cl} & \zero  \\
	\end{array}
	\right]
	&=&
	\left[
\begin{array}{c c | c}
	\A & -\B_2\Chat & \B_1 \\
	\Bhat\C_2 & \Ahat & \Bhat\D_{21} \\
	\hline
	\C_1 & -\D_{12}\Chat & \zero
\end{array}
\right].
\end{eqnarray*}
If the closed loop is observable and controllable and the plant satisfies the standard $\Htwo$ assumptions \cite{Geromel1997a}, then $||\Gcal_{cl}||_{\Htwo}^2 = \tr(\B_{cl}^T\Q\B_{cl})$, where the observability Gramian $\Q=\Q^T > 0$ solves $\A_{cl}^T\Q + \Q\A_{cl} + \C_{cl}^T\C_{cl} = \mbf{0}$ \cite{Dullerud2005}. Also, the CSL can enforce conic bounds on the controller through \autoref{dialated_conic}. 
 
 In the subsequent work, the transformation
\begin{eqnarray*}
	& \! \! \! \! \! \! \! \! \!
	\At = \bmat{\A & \mbf{0} \\ \mbf{0} & \mbf{0}},  \;
	\Bt = \bmat{\B_1 \\ \mbf{0}},  \;
	\Ct^T = \bmat{\C_1^T \\ \mbf{0}}, \;
	\mbf{E} = \bmat{-\B_2 & \mbf{0} \\ \mbf{0} & \I}, 
	&
	\\ 
	& \! \! \! \! \! \! \! \! \!
	\mbf{R} = \bmat{\C_2 & \mbf{0} \\ \mbf{0} & \I}, \;
	\mbf{S} = \bmat{\D_{21} \\ \mbf{0}}, \;
	\mbf{F}^T = \bmat{-\D_{12}^T \\ \mbf{0}}, \;
	\mbf{K} = \bmat{\mbf{0} & \Chat \\ \Bhat & \Ahat},
	&
	\\ 
	& \! \! \! \! \! \! \! \! \!
	\mbf{X}^T = \bmat{\mbf{0} & \I  \\ \I & \mbf{0} \\ \mbf{0} & \mbf{0}}, \; \Pt = \bmat{\mbf{0} & \bP \\ \mbf{0} & \mbf{0} \\ \I & \mbf{0}}, \; \Gamma = \bmat{\mbf{0} & * & * \\ \mbf{0} & -\frac{(a_c-b_c)^2}{4b_c}\I & * \\ \mbf{0} & -\frac{a_c+b_c}{2}\I  & -b_c\I},
	&
\end{eqnarray*}
which is an extension of the transformation used in \cite{Oliveira2000}, will greatly simplify notation. This transformation is also beneficial to some secondary objectives, like sparsity promotion \cite{Eilbrecht2017}, because it collects all controller parameters into a single matrix, $\K$. Now, the closed-loop state space realization and CSL inequality can be expressed as $\A_{cl} =  \At + \E\K\R$, $\B_{cl} =  \Bt + \E\K\bS $, $\C_{cl} =	\Ct + \F\K\R$, and $\He(\Pt\K\X) + \Gamma \leq 0$, and Problem~\ref{general} can be equivallently expressed
\begin{eqnarray} \label{KQP}
\!\!\!\!\!\!	\min_{\K,\Q,\bP} && \! \! \! \! \! \! \! \! \! \! \! \!  
	J(\K,\Q) = \tr((\Bt+\E\K\bS)^T\Q(\Bt+\E\K\bS)) \label{eqn:H2obj}\\
\!\!\!\!\!\!	\st && \! \! \! \! \! \! \! \! \! \! \! \! 
	\He[\Q(\At + \E\K\R)] + (\Ct+\F\K\R)^T(\Ct+\F\K\R) = \zero \label{eqn:lyapunov} \\
	&& \! \! \! \! \! \! \! \! \! \! \! \!  
	\He[\Pt\K\X] + \Gamma \leq 0 \label{eqn:conic} \\
	&& \! \! \! \! \! \! \! \! \! \! \! \!  
	\Q > 0, \qquad \bP > 0.  \label{eqn:P}	
\end{eqnarray}


\section{Main Result} \label{mainresult}

\subsection{Design Algorithm}
In this section, \autoref{main_theorem} establishes a convex optimization that minimizes an upper bound on the closed-loop $\Htwo$ norm subject to a conic controller constraint. The upper bound is nonconservative near the initialization point. Algorithm~1 uses \autoref{main_theorem} to iteratively reduce the $\Htwo$ norm from a feasible point to a local minimum. Corollary~\ref{thm:decreasing_bounded} shows that Algorithm 1 generates a sequence of monotonic decreasing costs bounded below, and it derives a finite bound on the number of iterations to convergence.

\begin{theorem} \label{main_theorem}
	Assume $\Q_0$, $\K_0$, and $\bP_0$ satisfy Equations~\ref{eqn:lyapunov}--\ref{eqn:P}. Then $\Q = \Q_0 + \dQ$, $\K = \K_0+ \dK$, and $\bP = \bP_0 + \dP$ also satisfy Equations \ref{eqn:lyapunov}--\ref{eqn:P}, where $\dQ$, $\dK$, and $\dP$ solve
	\begin{eqnarray} \label{deltaOpt2}
		\! \! \! \! \! \! \! \! \! \! \! \! \min_{\dQ,\dP,\dK,\mbf{Z}} &&  \! \! \! \! \! \! \! \! \! \! \!  J' = \tr(\Bt^T(\Q_0 + \dQ)\Bt) + \tr(\mbf{Z}) \label{main_obj} \\
		\mbox{s.t.}&&\! \! \! \! \! \! \! \! \! \! \! 
		\bmat{\Q_0^{-1} - \Q_0^{-1}\dQ\Q_0^{-1} & * \\  \bS^T(\K_0 + \dK)^T\E^T & \mbf{Z}} \geq 0  \label{c2} \\
		&& \! \! \! \! \! \! \! \! \! \! \!
		 \bmat{\Pi_1 & \dQ\E & \R^T\dK^T \\ * & -\W_1^{-1} & \mbf{0} \\ * & * & -(\W_1^{-1} + \F^T\F)^{-1}} \leq 0  \label{c1}\\
		&& 	\! \! \! \! \! \! \! \! \! \! \! 
		\bmat{\Pi_2 & \dPt & \X^T\dK^T \\ * & -\W_2^{-1} & \mbf{0} \\  * & * & -\W_2} \leq 0 \label{c3} \\
		&& \! \! \! \! \! \! \! \! \! \! \! 
		\Q_0 + \dQ > 0, \qquad \bP_0 + \dP > 0. \label{c5}
	\end{eqnarray}
and
\begin{eqnarray*}
	\Pi_1 &=& \He[\At^T\dQ + \At^T\Q_0 + \Q_0\E\K_0\R + \dQ\E\K_0\R     \\
	&&+ \Q_0\E\dK\R + \Ct^T\F\K_0\R + \Ct^T\F\dK\R \\
	&& + \R^T\K_0^T\F^T\F\dK\R] + \Ct^T\Ct + \R^T\K_0^T\F^T\F\K_0\R, \\
	\Pi_2 &=& \Gamma + \He[\Pt_0\K_0\X + \Pt_0\dK\X + \dPt\K_0\X], \\
	&&\W_1 = \W_1^T > 0 \qquad \mbox{and} \qquad \W_2 = \W_2^T  > 0.
\end{eqnarray*}
	Further, letting $\Q_{opt}$ and $\K_{opt}$ be the global minimizers of Problem~\ref{KQP}, it is true that
	\begin{eqnarray}
		J'(\K,\Q) &\geq& J(\K,\Q), \label{costConserve}\\
		J'(\K_0,\Q_0) &=& J(\K_0,\Q_0), \label{equality}\\
		 J'(\K_0,\Q_0)  &\geq& J'(\K,\Q), \label{upper_bound} \\
		 J(\K_{opt},\Q_{opt}) &\leq& J'(\K,\Q). \label{lower_bound}
		\end{eqnarray}
\end{theorem}
\begin{proof}
	The proof follows the overbounding technique from \cite{Warner2017a}, begining by converting Objective~\ref{eqn:H2obj} into a convex function with an added convex constraint, \autoref{c2}. Then, ICO transforms Equations~\ref{eqn:lyapunov},~\ref{eqn:conic} into Equations~\ref{c1},~\ref{c3}, which are LMIs in $\dQ$, $\dK$, $\dP$. Last, it is shown that these transformations imply Equations~\ref{costConserve}-\ref{lower_bound}.
	
	To begin, the objective in \autoref{eqn:H2obj} is distributed as $\tr(\Bt^T\Q\Bt) + \tr(\He[\Bt^T\Q\E\K\bS]) +  \tr(\bS^T\K^T\E^T\Q\E\K\bS)$. The first term is convex in a single design variable. Using the cyclic property of the trace, the second term is equivallently $\tr(\He[\D_{21}\B_1^T\Q^{12}\Bhat])$, where $\Q^{12}$ is the top-right block of $\Q$ when paritioned appropriately. This term is zero due to the standard assumption that $\D_{21}\B_1^T = \zero$. The third term is cubic in $\Q$ and $\K$. Consider a new variable $\Z$ such that $\Z - \bS^T\K^T\E^T\Q\E\K\bS \geq 0$. Minimizing $\tr(\Z)$ subject to this inequality is equivallent to minimizing the cubic term. Applying the Schur complement to the constraint yields an LMI in $\Z$, $\K$, $\Q^{-1}$. For any $\Qt>0$, $\Q>0$, it is known that
	\begin{eqnarray}
		\Q^{-1} \geq 2\Qt^{-1} - \Qt^{-1}\Q\Qt^{-1}. \label{conservative}
	\end{eqnarray}
	Applying this fact to the Schur complement yields
	\begin{eqnarray} \label{bound}
		\bmat{2\Qt^{-1} - \Qt^{-1}\Q\Qt^{-1} & * \\ (\E\K\bS)^T & \Z } \geq 0, \label{c2_almost}
	\end{eqnarray}
	which is linear in $\Q$, $\K$, and $\Z$. Now minimizing $\tr(\Z)$ is conservative, which implies \autoref{costConserve}.
	
	\autoref{eqn:lyapunov} is relaxed to a negative semi-definite inequality because the optimal solution converges to the equality boundary and negative semi-definiteness is sufficient to guarantee stability and a bounded solution domain. This bilinear inequaliy is treated with ICO \cite{Warner2017a}. Let $\K = \K_0 + \dK$ and $\Q = \Q_0 + \dQ$, where the new design variables $\{\dQ,\,\dK\}$ represent change from a known feasible point, $\{\K_0,\,\Q_0\}$. These substitutions yield
$
		 \Pi_1 + \He[\dQ\E\dK\R] + \R^T\dK^T\F^T\F\dK\R \leq 0
$
	where $\Pi_1$ is linear in the new variables. Next,
	\begin{eqnarray*}
		& \! \! \! \! \! \! \! \! \! (\dQ\E\bL - \R^T\dK^T\bL^{-T})(\dQ\E\bL - \R^T\dK^T\bL^{-T})^T =& \\
		& \! \! \! \! \! \! \! \! \! \dQ\E\bL\bL^T\E\E^T\dQ^T + \R^T\dK^T\bL^{-T}\bL^{-1}\dK\R + \He[\dQ\E\dK\R]&
	\end{eqnarray*}
	for any invertable $\bL$, which implies
	$
		\He[\dQ\E\dK\R] \leq \dQ\E\W_1\E^T\dQ^T + \R^T\dK^T\W_1^{-1}\dK\R
	$
	for any $\W_1 = \W_1^T>0$. This allows \autoref{eqn:lyapunov} to be rewritten conservatively as
	$
		\Pi_1 +\dQ\E\W_1\E^T\dQ^T + \R^T\dK^T(\W_1^{-1}+\F^T\F)\dK\R \leq 0.
	$
	Now the quadratic terms are removed using the Schur complement, which yields directly \autoref{c1}. An identical approach can be applied to \autoref{eqn:conic}, which becomes \autoref{c3} with $\W_2 = \W_2^T > 0$, $\bP = \bP_0 + \dP$, and
	\begin{eqnarray*}
		\Pt = \Pt_0 + \dPt = \bmat{\zero & \bP_0 \\ \zero & \zero \\ \I & \zero} + \bmat{\zero & \dP \\ \zero & \zero \\ \zero & \zero}.
	\end{eqnarray*}

	Applying the same change of variables in $\Q$ and $\K$ to \autoref{c2_almost} and selecting $\Qt = \Q_0$ yields \autoref{c2}.
	
	Thus any $\dK$, $\dQ$, $\dP$, $\Z$ satisfying Equations~\ref{c2}--\ref{c5} implies $\K$, $\Q$, $\bP$ satisfy Equations~\ref{eqn:lyapunov}--\ref{eqn:P} and the costs are related by Equation~\ref{costConserve}. Further, if $\dK=\zero$, $\dQ=\zero$, $\dP=\zero$, then Equations~\ref{c1}--\ref{c5} reduce to Equations~\ref{eqn:lyapunov}--\ref{eqn:P}, and Equations~\ref{main_obj}--\ref{c2} reduce to Equation~\ref{eqn:H2obj}. This implies Equation~\ref{equality}. Since $\K_0$, $\bP_0$, and $\Q_0$ are feasible points, the cost given by Problem~\ref{main_obj}--\ref{c5} cannot be worse than the initial cost, hence \autoref{upper_bound}. \autoref{lower_bound} follows from the definition of the minimum and the conservatism introduced by \autoref{conservative} and ICO.
\end{proof}

Due to the properties of Problem~\ref{deltaOpt2}--\ref{c5} found in Theorem~\ref{main_theorem}, Algorithm 1 gives a conic controller with a lower $\Hcal_2$-norm than that of the initialization. This is formalized in Corollary 1. Initializations are considered in the sequel.

\begin{algorithm}\label{main_alg}
	\SetAlgoLined
	\textbf{Input:} $\Gcal$, $a_c$, $b_c$, $\K_0$, $\Q_0$, $\bP_0$
	\\
	
	\textbf{Set }$k = 0$, $\Delta J' = \infty$\\
	\While{$\Delta J' > \epsilon$}
	{
		Solve Problem \ref{deltaOpt2}--\ref{c5} with $\K_0{=}\K_k$, $\Q_0{=}\Q_k$, $\bP_0{=}\bP_k$
		\\
		$\Q_{k+1} = \Q_{k} + \dQ$ \\
		$\bP_{k+1} = \bP_{k} + \dP$ \\
		$\K_{k+1} = \K_{k} + \dK$ \\
		$\Delta J' = |J'(\K_{k+1},\Q_{k+1}) - J'(\K_k,\Q_k)| $ \\
		k = k+1 
	}
	\textbf{Output:} $\K_{k+1}$, $\Q_{k+1}$, $\bP_{k+1}$
	\caption{}\label{alg:iter}
\end{algorithm}
\begin{cor} \label{thm:decreasing_bounded}
	At iteration $k$ of Algorithm 1, Problem \ref{deltaOpt2}--\ref{c5} has a feasible solution whose cost, $J'(\K_{k},\Q_{k})$ is bounded above by the cost at the $(k-1)^{\mbox{st}}$ iteration, $J'(\K_{k-1},\Q_{k-1})$,  and bounded below by the optimal cost of Problem~\ref{KQP}, $J(\K_{opt},\Q_{opt})$. Further, Algorithm 1 terminates in at most $(J'(\K_0,\Q_0)-J_{\Htwo})/\epsilon$ iterations, where $J_{\Htwo}$ is the cost of the $\Htwo$-optimal controller without the conic constraint.
\end{cor}


\subsection{Initialization} \label{initialization}

In this section, initialization schemes are proposed for the five parameters required to begin Algorithm 1. Two initialization schemes are proposed for $\W_1$ and $\W_2$ using various assumptions on $\dK$, $\dQ$, and $\dP$. Both attempt to minimize the conservatism introduced by $\W_1$ and $\W_2$. Meanwhile, $\K_0$, $\Q_0$, and $\bP_0$ must be initialized together. These have the greatest impact on the solution as their initialization determines to which local minimum the problem will converge. An arbitrary initialization method is proposed for its simplicity, and its pitfalls are described. Then two heuristic methods are proposed which provide an initialization point that is ``close" to the optimal controller without conic bounds.

\subsubsection{W Initialization}
The only necessary condition on $\W_1$ and $\W_2$ is positive definiteness, so the simplest option is to set both to identity. In fact, this choice works well in practice, as demonstrated in \autoref{example}. This is because the conservatism introduced by $\W_1$and $\W_2$ is dictated by the size of $\dQ\E\W_1\E^T\dQ^T + \R^T\dK^T\W_1^{-1}\dK\R$ and  $\dPt\W_2\dPt + \X^T\dK^T\W_2^{-1}\dK\X$, respectively. Both terms are small when $\dK$, $\dQ$, and $\dP$ are small. An alternative heuristic seeks to minimize these terms directly, assuming that $\dK$, $\dQ$, and $\dP$ are of the same size and not arbitrarily small. Setting $\dK$, $\dQ$, and $\dP$ to identity, the ``size'' of the conservative terms may be defined as $\tr(\E\W_1\E) + \tr(\R^T\W_1^{-1}\R)$ and $\tr(\W_2) + \tr(\X^T\W_2^{-1}\X)$, respectively. These are minimized via
\begin{eqnarray}
	\!\!\!\min_{\W_1,\V_1} && \! \! \! \tr(\E\W_1\E^T) + \tr(\V_1) \; \; \; \st \; \; \; \bmat{\W_1 & \R \\ \R^T & \V_1} \geq 0, \label{W1init}\\
	\!\!\!\min_{\W_2,\V_2} && \! \! \! \tr(\W_2) + \tr(\V_2) \; \; \; \st \; \; \; \bmat{\W_2 & \X \\ \X^T & \V_2} \geq 0. \label{W2init}
\end{eqnarray}
These optimization problems are derived by the same process that yields \autoref{c2_almost} without applying \autoref{conservative}. Since $\W_1 = \I$ and $\W_2 = \I$ are feasible points for every permissable $\E$, $\R$, and $\X$, these problems are always solvable.

\subsubsection{KQP Initialization: Arbitrary}
The only requirement on $\K_0$, $\Q_0$, $\bP_0$ is that they constitute a feasible point of Problem \ref{KQP}--\ref{eqn:P}. Thus, any conic controller could form a basis for the initialization. Given a known conic controller, $\K_0$ is compiled directly from the parameters, the algebraic Riccati Equation \ref{eqn:lyapunov} is solved for $\Q = \Q_0$, and a feasible $\bP = \bP_0$ is selected subject to Equations \ref{eqn:conic} and \ref{eqn:P}.

In practice, the choice of controller is further restricted. If $\K_0$ is small, then $\dK$ must be much smaller for Problem \ref{KQP}--\ref{eqn:P} to be near-nonconservative. In fact, if $\K_0 = \zero$ (which is trivially interior conic for all $a_c$, $b_c$) then any $\dK \neq \zero$ is infinitely conservative. In this case, Algorithm 1 immediately converges to the initialization point. Thus a sufficiently large $\K_0$ must be selected for good results, which complicates the already nontrivial task of constructing a conic controller.

\subsubsection{KQP Initialization: ConicC}
Even if an arbitrary controller can be directly constructed, it may not yield good results. While no initialization is guaranteed to achieve the global minimum due to the problem's nonconvexity, heuristic methods aim to select an initialization point that is in some sense already ``close" to the optimum. One such method is the ConicC algorithm \cite{Bridgeman2014}. This algorithm starts with the Luenberger optimal controller and changes $\Chat$ minimally to satisfy the conic bounds. The algorithm already calculates $\bP_0$,  and $\K_0$ is constructed from the resultant controller. Solving the algebraic Riccati Equation \ref{eqn:lyapunov} then provides $\Q_0$. Any other conic design method can be similarly employed. The ConicC problem is always feasible given a stable target controller. This is because the problem is convex and Lemma \ref{ConicClemma} below gives a feasible point. The original ConicC formulation uses the $\mathcal{H}_2$-optimal Luenberger controller as the target. While the Luenberger controller is usually open-loop stable, it is not always \cite{Johnson1979}. Nonetheless, a different target controller can be selected arbitrarily or by a variety of stable compensator designs \cite{MacMartin1994} to recover feasibility of the ConicC algorithm.

\begin{lemma} \label{ConicClemma}
	All stable LTI systems without feedthrough are interior conic for some $a<0<b$. Further, $\forall \; \mathcal{G}:(\A,\B,\C,\zero) \in \cone(a_0,b_0) \nsubseteq \cone(a,b)$ where $a_0,a<0$ and $b_0,b >0$, $\mathcal{G}':(\A,\B,\C',\zero) = (\A,\B,\min(\nicefrac{a}{a_0},\nicefrac{b}{b_0})\C,\zero) \in \cone(a,b)$. 
\end{lemma}
	\begin{proof}
		By the Lyapunov Stability Theorem \cite{Dullerud2005}, since $\A$ is Hurwitz, $\forall$ $\Q$, $\exists$ $\bP=\bP^T>0$ such that $\bP\A + \A^T\bP <-\Q$. In particular, $\Q = \C^T\C+\B\B^T$ can be chosen. If $\gamma  > \max|\lambda(\bP)|$, then $\nicefrac{1}{\gamma^2}\bP\B\B^T\bP < \B\B^T$. The Lyapunov equation then implies \autoref{contracted_conic} holds with $a = -\gamma$, $b = \gamma$. Thus by Lemma~\ref{CSL}, $\mathcal{G} \in \cone(-\gamma,\gamma)$.

		 Further, if $\mathcal{G} \in \cone(a,b)$ with $b>0$ and $k \geq 0$, then $k\mathcal{G} \in \cone(ka,kb)$ \cite{Zames1966}. Scaling $\mathcal{G}:u\rightarrow y$ is achieved by scaling the output matrix, $\C$. Suppose $a_0<a$ and $b_0>b$. 	In this case, $\nicefrac{a}{a_0}<1$ and $\nicefrac{b}{b_0}<1$, so $\min(\nicefrac{a}{a_0},\nicefrac{b}{b_0})a_0 > a$ and  $\min(\nicefrac{a}{a_0},\nicefrac{b}{b_0})b_0 < b$. Thus, $\mathcal{G}' = \min(\nicefrac{a}{a_0},\nicefrac{b}{b_0})\mathcal{G} \in \cone(a,b)$. Further, if $a_0>a$, then $\min(\nicefrac{a}{a_0},\nicefrac{b}{b_0}) = \nicefrac{b}{b_0}$, and $\frac{b}{b_0}a_0 > a$, and vice versa for $b_0<b$. Thus any interior conic system can be transformed to satisfy any interior conic bounds of opposing signs by scaling the output matrix.
\end{proof}
\subsubsection{KQP Initialization: ICO}
An alternative method, based on a technique in \cite{Warner2017a}, iteratively relaxes the optimal controller until it statisfies the constriants. First, the Luenberger optimal controller (or similar) is found, which provides $\K_L$. Next, Equation \ref{eqn:lyapunov} is solved with $\K_L$ to get $\Q_L$. Then $\bP_L$ is given by minimizing $\epsilon$ over $\bP_L$ and $\epsilon$ subject to $\He[\Pt_L\K_L\X] + \Gamma \leq \epsilon \I$ and $\bP_L > 0$. If $\epsilon<0$, then $(\K_L,\Q_L,\bP_L)$ is a feasible point, and the process is terminated. Otherwise, minimize $\epsilon$ over $\dK$,  $\dQ$, $\dP$, $\Z$ subject to Equations \ref{c2}--\ref{c5}, where $\K_0 = \K_L$, $\Q_0 = \Q_L$, $\bP_0 = \bP_L$, and $\leq 0$ is replaced by $\leq \epsilon\I$ in \autoref{c3}. Also add the constraint $\tr(\Bt^T(\Q_L+\dQ)\Bt) + \tr(\Z) < (1+\Delta)J_L$, where $\Delta$ is a small positive constant and $J_L$ is the cost of the initial controller. This extra constraint controls the rate at which the $\Hcal_2$ cost is relaxed. Iterate this new problem similarly to Algorithm 1, updating $J_L$ at each step, until $\epsilon<0$. If $\epsilon<0$ is achieved then the corresponding $\K$, $\Q$, $\bP$ is a feasible point. This algorithm is not guaranteed to converge, but it provides greater flexibility than ConicC when it does.


\section{Numerical Example} \label{example}

Algorithm 1 is now applied to a vibration suppresion problem with parametric uncertainty and passivity violation. The plant is three masses connected by springs and dampers. The $i^{th}$ mass' dynamics are given by $\x_i = [p_i \; v_i]^T$, $\dot\x_i = \A_{ii}\x_i + \sum_{j\neq i} \A_{ij}\x_j + [0 \; 1]^Tu_i + [0 \; 1]^Tw_{i1}$, $y_i = \C_{ii}\x_i + w_{i2}$, $\z_i = [p_i \; v_i \; u_i]^T$, $\w_i = [w_{i1} \; w_{i2}]^T$, where $p_i$ and $v_i$ are the position and velocity of mass $i$, $\A_{ii}$ and $\A_{ij}$ are derived from \autoref{chain}, and $\C_{ii}$ depends on the chosen output measurement. 

The idealized plant with velocity measured outputs and nominal parameters is $\mathcal{G}_1$. A more realistic version of the plant accounting for the parametric uncertainty and passivity violations is $\mathcal{G}_2$. Parametric uncertainty is introduced informally by considering the discrete set of parameters. Parameters $m_i$ and $k_i$ have nominal value 1 and off-nominal values 0.3 and 3, while $c_i$ has nominal value 0.05 and off-nominal values 0.01 and 0.1. This is not a full treatment of parametric uncertainty, merely a proof of concept. To introduce a passivity violation, the measured output is the position of each mass filtered by the approximate derivative filter $f(s) = \frac{25s}{s^2+4s + 25}$. This filter causes a modest increase to the minimal $\Hcal_2$ norm but significant passivity violations.

All controllers in this example are designed for the state-space model of $\mathcal{G}_1$ and the conic sector of $\mathcal{G}_2$, then tested on $\mathcal{G}_2$. Conic sector $(-24.84, 62200)$ is sufficient for the discrete set of parameters considered here. More analysis would be required to show that the sector is sufficient to describe all parameter sets in the convex polytope.

Two fixed-order controllers, \textit{Cnew} and \textit{Inew} are designed using Algorithm 1 with the proposed ConicC and ICO initializations, respectively. Both use the proposed \textit{W} initializations, which interestingly converge to the identity matrices in each case. The convergence criterion on Algorithm 1 is $\epsilon = 5\times10^{-3}$. The ICO initialization uses increment $\Delta = 0.1$. Algorithm 1 and the ICO initialization experienced numerical issues, but these were remedied by regularizing the design variables by augmenting the objective with the squared Frobenius norms of $\delta\K$, $\dP$, and/or $\delta\Q$ times a small constant $\gamma$. For Algorithm 1, $\gamma = 0.1$ was successful, and $\gamma = 10^{-3}$ was successful for the ICO initialization.

The novel \textit{Cnew} and \textit{Inew} controllers' performance is compared to several benchmark controllers: the $\mathcal{H}_2$\textit{-Optimal} controller as designed for $\mathcal{G}_1$ (which is passive), the \textit{ConicC} controller \cite{Bridgeman2014}, the \textit{Iterative Conic} controller \cite{Bridgeman2019a} with convergence criterion $\epsilon=10^{-6}$, and the open loop with no controller. \autoref{costTable} shows that \textit{Cnew} and \textit{Inew} significantly improve over the \textit{ConicC} and \textit{Iterative Conic} controllers for the nominal parameters. Interestingly, the resulting controllers are substantially different even though their performance is similar. \autoref{histogram} shows that the conic controllers are stable for every parameter set, whereas the (passive) $\mathcal{H}_2$-Optimal controller is unstable for almost 40\% of the simulations. Further, out of the five controllers that stabilize every simulation, \textit{Cnew} performed the best for about 60\% of simulations, while \textit{Inew}, performed the best for about 40\%. \textit{ConicC} performed about 10\% worse on all simulations, with \textit{Iterative Conic} slightly worse still. The design curves in \autoref{designcurves} emphasize that \textit{Inew} required an order of magnitude more iterations and converges to a slightly worse performance than \textit{Cnew}.

 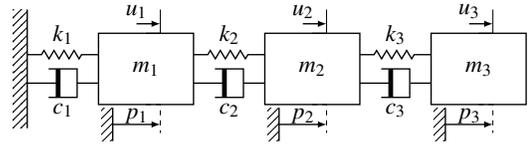
\begin{figure}
	\centering
	\resizebox{.8\columnwidth}{!}{%
		\begin{tikzpicture}[every node/.style={outer sep=0pt},thick,
 mass/.style = {draw,thick},
 spring/.style = {thick,decorate,decoration={zigzag,pre length=0.3cm,post length=0.3cm,segment length=6}},
 ground/.style = {fill,pattern=north east lines,draw=none,minimum width=0.75cm,minimum height=0.3cm},
 dampic/.pic={
 \fill[white] (-0.1,-0.3) rectangle (0.3,0.3);
 \draw (-0.3,0.3) -| (0.3,-0.3) -- (-0.3,-0.3);
 \draw[line width=1mm] (-0.1,-0.3) -- (-0.1,0.3);
 },
roundnode/.style={circle, draw=black, fill=black, scale = 0.25, minimum size=1mm}
 ]

  \node[mass,minimum width=2cm,minimum height=1.5cm] (m1) {\Large $m_1$};
  \node[mass,minimum width=2cm,minimum height=1.5cm,right=1.5cm of
  m1] (m2) {\Large $m_2$};
   \node[mass,minimum width=2cm,minimum height=1.5cm,right=1.5cm of
  m2] (m3) {\Large $m_{3}$};
  
   \node[left=1.5cm of m1,ground,minimum width=3mm,minimum height=2.5cm] (g1){};
  \draw (g1.north east) -- (g1.south east);

  \draw[spring] ([yshift=3mm]g1.east) coordinate(aux)
   -- (m1.west|-aux) node[midway,above=1mm]{\Large $k_1$};
  \draw[spring]  (m1.east|-aux) -- (m2.west|-aux) node[midway,above=1mm]{\Large $k_2$};
  \draw[spring]  (m2.east|-aux) -- (m3.west|-aux) node[midway,above=1mm]{\Large $k_3$};

  \draw ([yshift=-3mm]g1.east) coordinate(aux')
   -- (m1.west|-aux') pic[midway]{dampic} node[midway,below=3mm]{\Large$c_1$}
     (m1.east|-aux') -- (m2.west|-aux') pic[midway]{dampic} node[midway,below=3mm]{\Large $c_2$}
     (m2.east|-aux') -- (m3.west|-aux') pic[midway]{dampic} node[midway,below=3mm]{\Large $c_3$};

  \foreach \X in {1,2,3}  
  {\draw[thin] (m\X.north) -| ++ (0.3,0.5) coordinate[pos=.7](aux\X);
   \draw[latex-] (aux\X) -- ++ (-0.5,0) node[above]{\Large $u_\X$}; 
   \draw[thin,dashed] (m\X.south) -| ++ (0.3,-0.6) coordinate[pos=0.85](aux'\X);
   \draw[latex-] (aux'\X) -- ++ (-1,0) node[midway,above,yshift=-1mm]{\Large $p_\X$}
    node[left,ground,minimum height=7mm,minimum width=1mm] (g'\X){};
   \draw[thick] (g'\X.north east) -- (g'\X.south east);
  }

  \end{tikzpicture}
	}
	\caption{Example system}
	\label{chain}
\end{figure}

\begin{table}
	\centering
	\begin{tabular}{|c| c c c c|} 
		\hline
		 & \makecell{Cost on \\ nominal $\mathcal{G}_2$} & \makecell{\% increase \\ from Optimal \\ for $\mathcal{G}_2$} & \makecell{\# of \\ Iter.} & \makecell{In \\ cone?} \\
		\hline
		 $\Hcal_2$-Optimal for $\mathcal{G}_1$  & 12.98 & 2\% & n/a & no \\
		  ConicC \cite{Bridgeman2014}  & 77.98 & 514\% & n/a & yes \\
		  Iterative Conic \cite{Bridgeman2019a}  & 73.74 & 480\% & 17 & yes \\
		  Inew & 72.83 & 473\% & 689 & yes \\
		  Cnew & 72.10 & 467\% & 34 & yes \\
		\hline
	\end{tabular}\caption{Nominal closed-loop $\Hcal_2$-norm comparisons.}\label{costTable}
\end{table}%

\begin{figure}[h!]
	\setlength\fwidth{.75\columnwidth}
	\centering
	\input{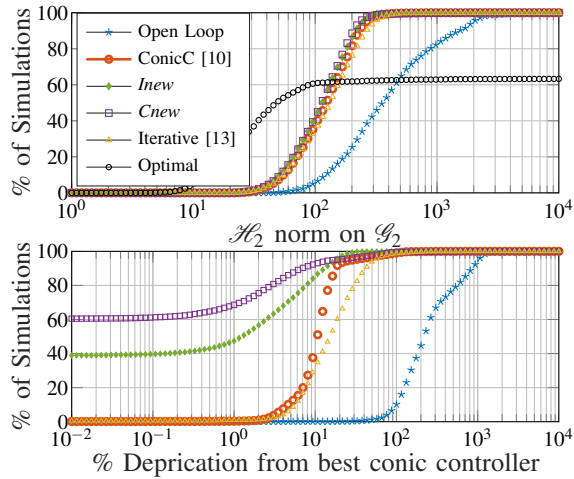}
	\caption{\textsc{Top:} Percent of simulations with closed-loop $\Hcal_2$ norm below a given cost. \textsc{Bottom:} Percent of simulations with a closed-loop $\Hcal_2$ norm that differs from the best conic controller (out of the five) by less than a given percent error.}
	\label{histogram}
\end{figure}

\begin{figure}
	\setlength\fwidth{1.2\columnwidth}
	\centering
	\input{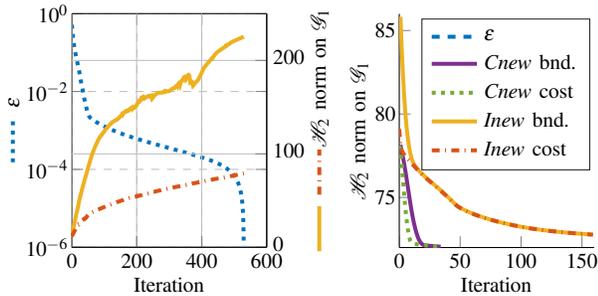}
	\caption{\textsc{Left:} ICO initialization. \textsc{Right:} Algorithm 1 for \textit{Cnew} and \textit{Inew}. Both plot the cost overbound calculated by the algorithms and the actual cost on the idealized system.}
	\label{designcurves}
\end{figure}


\section{Discussion} \label{conclusion}

Motivated by the need to incorporate secondary objectives like structured communication into controller design, this paper develops a fixed-order $\Hcal_2$-conic synthesis that allows direct access to design variables. A convergent series of convex problems are posed that minimize the $\Hcal_2$ norm while maintaining desired conic bounds without restricting design freedom or transforming controller parameters into implicit variables. The method's conservatism is concentrated mostly in the choice of initialization, and initialization schemes are proposed based on successful heuristics. 

Ultimately, the analytical results herein are intermediates to improved sparsity promoting and structured $\Hcal_2$-conic designs. While those results are immenant, they remain nontrivial. Thus, complete development of those new designs are left as the subject of future work. Other future work includes extending these methods to more general IO stability results like dissipativity. Nonetheless, this design presents immediate advantages. Numerical simulations show that the increased design freedom significantly improves performance compared to existing $\Hcal_2$-conic designs when applied to an example system with passivity violation and uncertain parameters. These simulations also emphasize the value of $\Hcal_2$-conic control in general, as all of the conic designs considered guranteed stability for every parameter set, while the (passive) optimal controller was unstable for about 40\%.

\addtolength{\textheight}{0cm}   







\bibliographystyle{IEEEtran}
\bibliography{IEEEabrv,root}

\end{document}